\documentclass[a4paper,11pt]{article}

\usepackage{geometry}
\RequirePackage{amsthm,amsmath,amsfonts,amssymb}
\RequirePackage[numbers]{natbib}

\RequirePackage{hyperref}
\RequirePackage{mathrsfs}
\RequirePackage{bbm}
\RequirePackage{units}
\RequirePackage{multirow}
\RequirePackage{makecell}

\theoremstyle{plain}
\newtheorem{theorem}{Theorem}[section]
\newtheorem{corollary}[theorem]{Corollary}
\newtheorem{proposition}[theorem]{Proposition}

\newtheorem{lemma}[theorem]{Lemma}
\newtheorem{conjecture}[theorem]{Conjecture}
\newtheorem*{theorem*}{Theorem}
\theoremstyle{remark}

\newtheorem{remark}[theorem]{Remark}
\newtheorem{definition}[theorem]{Definition}

\DeclareMathOperator{\p}{\mathbb{P}}

\newcommand{\cE}{\mathcal{E}}
\newcommand{\cF}{\mathcal{F}}
\newcommand{\cO}{\mathcal{O}}
\newcommand{\cU}{\mathcal{U}}
\newcommand{\cL}{\mathcal{L}}

\newcommand{\de}{\operatorname{d}}

\newcommand{\sC}{\mathscr{C}}
\newcommand{\sG}{\mathscr{G}}

\usepackage{orcidlink}
\usepackage[affil-sl]{authblk}
\title{The Phase Transition of the Voter Model\\ on Evolving Scale-Free Networks}
\author{John Fernley \orcidlink{0000-0002-6635-4341}}
\affil{Department of Statistics,
University of Warwick,
United Kingdom}
\date{\today}

\renewcommand\labelenumi{(\roman{enumi})}
\renewcommand\theenumi\labelenumi

\begin{document}

\maketitle

\begin{abstract}
The voter model is a classical interacting particle system explaining consensus formation on a social network. Real social networks feature not only a heterogeneous degree distribution but also connections changing over time. 
We study the voter model on a rank one scale-free network evolving in time by each vertex \emph{updating} (refreshing its edge neighbourhood) at any rate $\kappa=\kappa(N)$. 

We find the dynamic giant component phase transition in the consensus time of the voter model: when $\kappa\ll \nicefrac{1}{N}$, the subcritical graph parameters are slower by a factor of $\nicefrac{N}{\log N}$.
Conversely, when $\kappa \gg 1$ the effect of the giant is removed completely and so for either graph parameter case we see consensus time on the same order as in the static supercritical case (up to polylogarithmic corrections). The intermediate dynamic speeds produce consensus time for subcritical network parameters longer not by the previous factor $\nicefrac{N}{\log N}$, but by the factor $\nicefrac{1}{\kappa}$.

\vspace{1em}

{\footnotesize \noindent \emph{Keywords:} consensus dynamics; heterogeneous network; dynamic graph; power-law degree.

}
\end{abstract}

\section{Introduction}

The voter model was introduced to model competing species in space by \cite{clifford1973model}, and in the mathematical literature for its interest as an infinite system \cite{holley1975ergodic}, albeit with the suggestive name ``voter''. If we look to understand the phenomenology of consensus formation through the simplest possible model, this is a likely point to arrive: put agents on a graph with one of two opinions and have them imitate a random neighbour iteratively until local consensus building forms a global consensus.

Both those foundational works put the voters on $\mathbb{Z}^d$ but in network science we want to consider finite graphs, moreover models of small diameter \cite{van2016random}, Section 1.6.1. On a fixed graph, given (as is widely believed) that a network model has faster order of mixing than meeting, its consensus time is well understood up to factors of the mixing time \cite{coxchenchoi}.

More realistically, networks change in time, less considered in the literature. Then in the dual model \cite{liggett85} we have coalescing walkers which no longer move independently, and in this article we investigate how the coalescence speed is affected.

On the vertex set $[N]:=\{1,\dots,N\}$ we put the evolving simple Norros--Reittu graph \cite{Norros_Reittu_2006} with vertex update rate $\kappa$. This dynamic graph has a stationary distribution of independent edges, each $\{i,j\}$ being present with probability
\begin{equation}\label{eq_nr}
p_{ij}=1-\exp\left(-
\beta N^{2\gamma-1}i^{-\gamma}j^{-\gamma}
\right).
\end{equation}

This expression is the simple graph version of the multigraph with Poissonian number of edges $\{i,j\}$ of expectation exactly $\beta N^{2\gamma-1}i^{-\gamma}j^{-\gamma}$, the multigraph then having $\de(i) \sim \operatorname{Pois}(w(i))$ where
\begin{equation}\label{eq_weight}
w(i)=
\frac{\beta}{N}\left(
\frac{N}{i}
\right)^\gamma
\sum_{j=1}^N
\left(
\frac{N}{j}
\right)^\gamma
\sim
\frac{\beta}{1-\gamma}\left(
\frac{N}{i}
\right)^\gamma.
\end{equation}

The Norros Reittu graph is part of the ``rank $1$'' equivalence class of network models, so called because the expectation of the adjacency matrix has rank $1$, including for example the configuration model. Mathematically however, it is more simple to have the exact independent Poisson distributions of the Norros-Reittu definition.

\begin{definition}
The event $(E_N)_N$ holds \emph{with high probability} if $\lim_{N\rightarrow\infty}\p(E_N)=1$.
\end{definition}

\begin{definition}
Graphs $(G_N)_N$, $(G'_N)_N$ are \emph{asymptotically equivalent} if they can be coupled with high probability: such that $\lim_{N\rightarrow\infty}\p(G_N=G'_N)=1$.
\end{definition}

After \emph{flattening} the previous multigraph construction to the largest simple graph it contains, we arrive at the expression \eqref{eq_nr}. Note for $\gamma=0$ this is asymptotically equivalent to an Erd\H{o}s--R\'enyi graph but otherwise for $\gamma \in (0,1)$ a sparse scale-free network. Where $\gamma<\nicefrac{1}{2}$ and so every edge probability tends to $0$, it is asymptotically equivalent to many other natural scale-free network models (see \cite{van2016random}, Theorem 6.18), most prominently the Chung--Lu model \cite{chung2002} where the edge probabilities are given by $\beta N^{2\gamma-1}i^{-\gamma}j^{-\gamma}$.

Because the model has independent edges, it is convenient to give it the vertex updating dynamics of \cite{jacob2017contact,morters2019,jacob2024metastability}, resulting in the following model.

\begin{definition}[The evolving scale-free network]
For any $\kappa \geq 0$, the dynamic graph $G_\kappa(t)$ is a Markov chain (in $\sG_N$ the space of graphs on $[N]$) with stationary distribution $\cL_{\rm SNR}$ of independent edges with the probabilities \eqref{eq_nr}.

The dynamic of the graph attaches a Poisson process of rate $\kappa$ to each vertex such that every ring of some $i \in [N]$ is a \emph{vertex update}: all edges incident to $i$ are removed and the $N-1$ incident potential edges are resampled independently according to \eqref{eq_nr}.
\end{definition}

There are other ways to produce a dynamic graph, in particular dynamical percolation as investigated in \cite{thomassousi} for the same symmetric generator walkers of this article. However, this dynamical percolation does not produce an inhomogeneous degree distribution. It would be possible to update edges at rate $\nicefrac{\kappa}{N}$ with the probability \eqref{eq_nr} for a model between ours and theirs as considered for the contact process in \cite{jacob2022contact}, but we consider vertex updating appropriate for social modelling in so far as it reflects agents moving themselves on the graph. We follow \cite{jacob2017contact,morters2019,jacob2024metastability} in calling this the \emph{evolving} graph as opposed to the \emph{dynamical} graphs where edges change completely independently.

If the above version of dynamical percolation were considered, we conjecture that the results would not change at all, as the graph mixing time has the same order in both cases and the effect of walkers walkers mixing in time $\Theta(\nicefrac{1}{\kappa})$ is also unaffected. However, the details of the argument would be not as clean as Lemma \ref{lem_sep} which only requires a single update, or as \cite{thomassousi}, Proposition 6.1, which exploits the symmetry of the homogenous graph.

We also want to model these agents sharing their opinions. 
For any $u \in (0,1)$ determining the bias in initial opinion, voter initial opinions are taken from the Bernoulli process with measure $\mu_u$ such that each $v \in [N]$ has independent $\operatorname{Bernoulli}(u)$ initial opinion $\mu_u(v)$. The proofs of Theorems \ref{theorem_static_cons} and \ref{theorem_dynamic_cons} could really be applied to any distribution of initial opinions a positive probability of $\Omega(N)$ differing opinions but we will not push this assumption.

\begin{definition}\label{def:voter}
Let $G_t = ([N],E_t)$ be a simple dynamic graph. Given $\eta \in \{ 0,1\}^N$, define $\eta^{i \leftarrow j} \in \{ 0,1\}^N$ as follows:
\[ \eta^{i \leftarrow j} (k) =
\begin{cases}
\eta(j), & k = i, \\
\eta(k), & k \neq i.
\end{cases}\]

The \emph{voter model} is the Markov process $(\eta_t)_{t \geq 0}$ with state space $\{0,1\}^N$ and with dynamic that for every directed pair of neighbours $\{i,j\} \in E_t$, the state $\eta \in \{0,1\}^V$ is replaced by $\eta^{i \leftarrow j}$ at rate $1$.
\end{definition}

The definition above, the process of interest in this article, is the voter model where all edges communicate opinion at rate $1$, dual to the variable-speed simple random walk (VSRW) with generator
\begin{equation}\label{eq_vsrw}
Q_{ij}=
\begin{cases}
\mathbbm{1}_{i \sim j}, & i \neq j,\\
-\de(i), & i=j.
\end{cases}
\end{equation}

This is also the original definition of the voter model due to \cite{clifford1973model}. 
We define the consensus time $T_{\rm cons}$ as the first time after which no vertex changes opinion. In the dynamic context of our main result Theorem \ref{theorem_dynamic_cons}, this will be a true consensus where some opinion in $\{0,1\}$ is shared by every vertex in $[N]$.

The usual voter model duality of \cite{liggett85} reverses time in the voter model, after which the path of the infection histories becomes a system of coalescing walkers. To be precise, by coalescing walkers we mean the model where every walker is indexed in $[N]$ by its initial site and then all walkers at a site move together to follow the walker of lowest index. From this duality and the fact that the stationary dynamic graph has the same distribution backward in time, we have the following result.

\begin{lemma}\label{lemma_coal_cons}
The dual system of coalescing walkers $(\xi_t)_t=(\xi^{(1)}_t, \dots, \xi^{(N)}_t)$ from $\xi_0=(1, \dots, N)$ has coalescence time $T_{\rm coal}$ when it hits a constant vector or meeting $T_{\rm meet}(i,j)$ when hitting the set with equal $i$\textsuperscript{th} and $j$\textsuperscript{th} elements. We find consensus in the stochastic interval
\[
T_{\rm meet} (i,j)
\mathbbm{1}_{\eta_0(i) \neq \eta_0(j)}
\preceq
T_{\rm cons}
\preceq
T_{\rm coal}.
\]
\end{lemma}

In applying the above we can make any random choice of $i$ and $j$, typically two stationary (uniform) walkers. 

For comparison with our main theorem we write the case of static graphs. 
Here for consistency we define the static graph as the graph process $\left(G_0(t)\right)_t$, but note that with the parameter $\kappa=0$ this graph is constant in time.

This result uses the notation $\Theta_{\mathbb{P}}^{\log N}(\cdot)$ to denote a two-sided order bound satisfied with high probability and allowing a poly-logarithmic correction factor. That is,
\begin{equation}\label{eq_polylog_prob}
f(N)=\Theta_{\mathbb{P}}^{\log N}(g(N))
\iff
 \exists C >0 : \mathbb{P}\left(
 \frac{g(N)}{\log^{2C} N}
 <
\frac{f(N)}{\log^C N}
<
g(N)
 \right)\rightarrow 1.
\end{equation}

\begin{theorem}[ \cite{fernleyortgiese,fernley2022discursive} ]\label{theorem_static_cons}
Take $\gamma\in \left[0,1\right)$, $u \in (0,1)$, and $\beta \neq 1-2\gamma$. Then for the voter model on $\left(G_0(t)\right)_t$ with initial conditions of a stationary graph and the opinion of each $v \in [N]$ from the Bernoulli process $\mu_u(v)$, we have
\[
\mathbb{E}_{\mu_u}\left(T_{\text{\textnormal{cons}}}\big|G_0(0)\right)=
\begin{cases}
\Theta_{\mathbb{P}}^{\log N}\left( N \right), &\beta +2\gamma>1,\\
\Theta_{\mathbb{P}}^{\log N}\left( N^\gamma\right), &\beta +2\gamma<1.
\end{cases}
\]
\end{theorem}

The two cases of the above result were covered for supercritical parameters in \cite{fernley2022discursive} and for subcritical parameters in \cite{fernleyortgiese}. These works looked at the whole family of ``discursive'' voter model dynamics parametrised by $\theta \in \mathbb{R}$, here we restrict our attention to the most popular model $\theta=1$. Neither previous work allowed the graph to move at all and in both cases, at $\theta=1$, the order above is linear in the size of the largest component. In the following section we will see how those times change when the graph does move.

\section{Main results}

Our main theorem still uses the polylogarithmic corrections of \eqref{eq_polylog_prob}, but on the dynamic graph there is no need to restrict to a good environment event and so the polylogarithmic factor holds deterministically  $\Theta^{\log N}\left( \cdot \right)$ rather than in probability $\Theta_{\mathbb{P}}^{\log N}\left( \cdot \right)$.

\begin{theorem}\label{theorem_dynamic_cons}
Take $\gamma\in \left[0,1\right)$, $u \in (0,1)$, and $\alpha \in \mathbb{R}$. Then, for any mean update time $\nicefrac{1}{\kappa}=\Theta(N^\alpha)$, consider the voter model on $\left(G_\kappa(t)\right)_{t\geq 0}$ from initial conditions:
\begin{itemize}
\item stationary graph $G_\kappa(0)\sim\cL_{\rm SNR}$;
\item Bernoulli process opinions $\eta_0(v)=\mu_u(v)$.
\end{itemize}

We have the following dichotomy: 
\begin{itemize}
\item for subcritical parameters $\beta+2\gamma<1$
\[
    \mathbb{E}_{\mu_u}\left(T_{\text{\textnormal{cons}}}\right)=
    \begin{cases}
    \Theta^{\log N}\left(  N\right), &  \alpha \leq 0,\\[5pt]
    \Theta\left( N^{1+\alpha} \right), &  \alpha > 0;
    \end{cases}
\]
\item for supercritical parameters $\beta+2\gamma>1$
\[
    \mathbb{E}_{\mu_u}\left(T_{\text{\textnormal{cons}}}\right)=
    \begin{cases}
    \Theta^{\log N}\left(  N\right), &  \alpha < 1,\\[5pt]
    \Theta\left( N^{\alpha} \log N\right), &  \alpha > 1  \text{ and } \beta\geq 3.
    \end{cases}
\]
\end{itemize}
\end{theorem}

\begin{proof}[Proof of Theorem \ref{theorem_dynamic_cons}]
The cases are broken up through this article as in Table \ref{table_proof}.

\begin{table}[b]
\caption{The cases of the proof of Theorem \ref{theorem_dynamic_cons}}\vspace{0.5em}
\label{table_proof}\centering
\bgroup
\def\arraystretch{1.5}
\begin{tabular}{cccc}
\hline
\multirow{3}{*}{Subcritical $\beta$, $\gamma$} &        Upper bound           & \multicolumn{2}{c}{Proposition \ref{prop_subcrit_coal}}                     \\ \cline{2-4} 
                  & \multirow{2}{*}{Lower bound} & \multicolumn{1}{c}{Slow updating} &   Proposition \ref{prop_coalescing_components}                \\ \cline{3-4} 
                  &                   & \multicolumn{1}{c}{Fast updating} & \multirow{2}{*}{Corollary \ref{corr_N_lower}} \\ \cline{1-3}
\multirow{4}{*}{Supercritical $\beta$, $\gamma$} & \multirow{2}{*}{Lower bound} & \multicolumn{1}{c}{Fast updating} &                   \\ \cline{3-4} 
                  &                   & \multicolumn{1}{l}{Slow updating} &     Proposition \ref{prop_last_singleton}              \\ \cline{2-4} 
                  & \multirow{2}{*}{Upper bound} & \multicolumn{1}{c}{Fast updating} &    Corollary \ref{cor_super_fast_upper}               \\ \cline{3-4} 
                  &                   & \multicolumn{1}{l}{Slow updating} &     Proposition \ref{prop_small_k_coal}              \\ \hline
\end{tabular}
\egroup
\end{table}

The lower bounds of Corollary \ref{corr_N_lower} and Proposition \ref{prop_coalescing_components} produce a pair of slow-meeting walkers, while in Proposition \ref{prop_last_singleton} to produce a $\log$ factor we lower bound consensus directly.
Then the upper bounds are for the coalescence time of $N$ walkers, initialised on every site. These contain the consensus time by Lemma \ref{lemma_coal_cons}.
\end{proof}

The natural conjecture follows by removing technical assumptions and polylogarithmic corrections.%

\begin{conjecture}
With the same set-up as Theorem \ref{theorem_dynamic_cons} except with arbitrary $\kappa$, we have
\[
    \mathbb{E}_{\mu_u}\left(T_{\text{\textnormal{cons}}}\right)=
    \begin{cases}
    \Theta\left(  N+\frac{N}{\kappa} \right), & \beta+2\gamma<1,\\[5pt]
    \Theta\left( N+\frac{1}{\kappa}\log N \right), & \beta+2\gamma>1.
    \end{cases}
\]
\end{conjecture}

Removing the $\beta \geq 3$ requirement would need a substantial work on the mixing times of these networks, which we sidestep in \cite{fernley2022discursive} by using an Erd\H{o}s--R\'enyi subgraph. It is widely expected that this is a technical assumption and these networks are generally mixing in polylogarithmic time. 
Removing the log corrections on the whole interval $\kappa \gg \nicefrac{1}{N}$ is more of a step, but it is easy to see that they can at least be removed for $\kappa\rightarrow \infty$ fast enough: the simplest argument as in the following proposition.

\begin{proposition}
With the same set-up as Theorem \ref{theorem_dynamic_cons} and $\alpha<-3$, we have
\[
    \mathbb{E}_{\mu_u}\left(T_{\text{\textnormal{cons}}}\right)=\Theta\left( N\right).
\]
\end{proposition}

\begin{proof}
$N$ walkers on a simple graph can move at total rate at most $N(N-1)$. At this point $\alpha<-3$, the updates at rate $\kappa$ are fast enough that we will see an update at every vertex between any two rings of a rate $N^2$ Poisson clock to time $O(N)$, with high probability. If this fails then restart the argument.

On this event, no edge is remembered between voter model events and so every walker can be thought of as constantly adjacent to every other with probability at least $p_{N,N}$ (recall \eqref{eq_nr}). This produces expected time to coalescence of one of $w$ walkers at most
\[
\left[
\frac{\beta+o(1)}{N}\binom{w}{2}
\right]^{-1}
\]
which sums to the claimed order. The corresponding lower bound is Corollary \ref{corr_N_lower}.
\end{proof}

\subsection{Discussion}

The same consensus time problem on the $\kappa=0$ graph is explored in \cite{moinet2018generalized}. Their approximation is by the \emph{heterogeneous mean field} where the network is replaced with a weighted complete graph. In the language of this article, this is precisely the $\kappa=\infty$ model. Indeed we do find linear consensus time in Theorem \ref{theorem_dynamic_cons} when $\alpha=-\infty$, which agrees with their result, and the $\kappa=0$ result of Theorem \ref{theorem_static_cons} agrees too if only on the supercritical network parameters. Moreover we can extend the heterogenous mean field result at $\alpha=-\infty$ to a larger phase $\{\alpha \leq 0\}$ through which the dynamic is fast enough that the consensus time order is not affected by the giant component, and the heterogenous mean field approximation is valid for all network parameters. Differing consensus times outside of this phase with subcritical graph parameters, and moreover the appearance of a log factor, were not predicted.

For the static model, the consensus orders of Theorem \ref{theorem_static_cons} correspond to the number of vertices in the largest component that must come to consensus. Instead in Theorem \ref{theorem_dynamic_cons} the $N$ factor that appears for $\kappa \ll \nicefrac{1}{N}$ is the difference in \emph{average} component size, and so there are no graph parameters in the exponents.

The limit $\kappa \downarrow 0$ or $\alpha \rightarrow \infty$ in Theorem \ref{theorem_dynamic_cons} produces a diverging mean consensus time.
 This is of course very different to the $\kappa=0$ times of Theorem \ref{theorem_static_cons}, reflecting that the $\kappa=0$ model does not really come to a full consensus but only componentwise consensus.
 
 The critical case $\beta+2\gamma=1$ we do not consider but nor is it covered by our methods: constant expectations in the lower bound of Proposition \ref{prop_coalescing_components} become polynomial in $N$, and a corresponding upper bound would likely need some detailed understanding of the critical dynamic structure beyond the component size established in \cite{zbMATH06691463, zbMATH07310499}, for example Theorem 3.5 of \cite{bhamidi2017continuum} covers the diameter for $\gamma$ small enough or Theorem 1.1 of \cite{christina2023} the typical distance in the large $\gamma$ case.
 
Foreseeable future work would not include other asymmetrical voting dynamics as it is hard to understand mixing of these models with a changing graph and hence a changing stationary distribution (see however \cite{hermon2020comparison}). Still, there is very much the possibility of looking at other graph dynamics. In particular, the site-dependent update rates of \cite{morters2019,jacob2024metastability}, where each vertex $i$ updates not at constant rate $\kappa$ but at rate $\kappa w(i)^\eta$ depending on its mean degree, might reintroduce the effect of the precise $\gamma$ of the degree distribution tail.

\begin{remark}[Related models]
As well as the previously discussed dynamical percolation, there is a way to construct a dynamic graph with stationary distribution of the configuration model \cite{trichotomy}. This makes the technicalities more difficult, in particular the mixing time order of the graph process has not been established, even in the case of constant degree \cite{zbMATH05228261}. 
Note when $\gamma<\nicefrac{1}{2}$ that if we put the degree sequence of our model into the dynamic configuration model then the two stationary distributions are asymptotically equivalent \cite{van2016random}, Theorem 6.15.

In the computer science literature, it is more common to consider the \emph{synchronous} voter model where voting is done in discrete time rounds as in \cite{zbMATH07753169}. This we opt against as in our social context it is not natural to persuade someone and simultaneously imitate them.
\end{remark}

\section{Lower bounds}

In this section we demonstrate the lower bounds on the expected consensus time. As was noted in Lemma \ref{lemma_coal_cons}, it suffices to find a pair of walkers with positive probability of not meeting in a time period of the claimed order. Still, as throughout, $N$ is taken sufficiently large in every statement.

Two independent walkers on $[N]$ and an independent graph process can be built into a single Markov chain with state
\begin{equation}\label{eq_M}
M_t=\left(
W_t^{(1)},
W_t^{(2)},
G_\kappa(t)
\right)
\in [N] \times [N] \times \sG_N
\end{equation}
where $\sG_N$ denotes the set of graphs on $[N]$. Note that this chain is reversible with stationary measure $\pi \otimes \pi \otimes \cL_{\rm SNR}$.

\begin{definition}\label{def_cL}
For a graph $g$ with vertex set $[N]$, write $\cL(g)$ for the SNR mass attached to $g$, defined by the independent edge probabilities \eqref{eq_nr}. For a graph functional $f$ we also write $\cL(f)$ to denote its expectation under $\cL$.
\end{definition}

We construct $M$ to understand meetings on the dynamic graph, which are hitting times of the the diagonal $D:=\{(i,i):i\in [N]\}\subset [N] \times [N]$.

\begin{definition}\label{def_rho}
The ergodic exit distribution $\rho$ is defined by sampling $\rho^-$ on $D \times \sG_N$ with law
\[
\rho^-(i,i,g) \propto \cL(g) \de_g(i)
\]
and then advancing one of the walkers a single step to a uniformly chosen neighbour (in $g$), each with probability $\nicefrac{1}{2}$.
\end{definition}

Note that the graph selected by $\rho^-$ of Definition \ref{def_rho} is actually size-biased in that it's drawn proportional to the total degree, however in the following lemma we see that this is a minor distinction in that the two models can be coupled with high probability.

\begin{lemma}\label{lem_asymptotic_equiv}
The size-biased version $\cL^*$ of $\cL$ is asymptotically equivalent to $\cL$.
\end{lemma}

\begin{proof}
Edges in the stationary measure $\cL$ are simply independent Bernoulli variables and so by a standard Chernoff bound \cite{mcdiarmid1998concentration}, Theorem 2.3(a)
\[
\p_{\cL}\left(
\left|
\frac{\de_g([N])}{\cL\left(\de_g([N])\right)}-1
\right|
>
\delta
\right)
\leq
2e^{\nicefrac{-2\delta^2\cL\left(\de_g([N])\right)^2}{N}}
=e^{-\Omega(\delta^2 N)}.
\]

On the complement set, then, a size bias
\[
\cL^*(g)=\frac{\de_g([N])}{\cL(\de_g([N]))}\cL(g)
\]
can only %
reduce the mass of a graph by a factor $\tfrac{2\delta}{1+\delta}$. So we fail to couple if the stationary graph falls in the large deviation set or conditionally with the complement of this factor, i.e. with probability
\[
\frac{2\delta}{1+\delta} + e^{-\Omega(\delta^2 N)}
=O\left(
\frac{\log N}{\sqrt{N}}
\right)
\]
by setting $\delta=\frac{C\log N}{\sqrt{N}}$ with a sufficiently large $C$.
\end{proof}

Much more approachable than meeting of two (hidden) Markov chains depending on the same Markov graph, we consider hitting of the single Markov chain $M$ of \eqref{eq_M}. By applying Kac's formula \cite{aldous-fill-2014}, Equation 2.24, we immediately obtain the following result for the typical return time.

\begin{lemma}\label{prop_ergodic_return}
\[
\mathbb{E}_\rho \left( T_D \right)
=\frac{1-\tfrac{1}{N}}{\tfrac{2}{N^2}\sum_{i=1}^N \mathbb{E}_{\rm SNR}(\de(i))}
\sim 
\left(
\frac{1-\gamma}{2\beta}
\right)N.
\]
\end{lemma}

We then have an analogue to \cite{coxchenchoi}, Corollary 3.4, in the dynamic graph.

\begin{corollary}\label{corr_N_lower}
\[
\mathbb{E}_\pi \left( T_D \right) =\Omega(N).
\]
\end{corollary}

\begin{proof}
From \cite{aldous-fill-2014}, Proposition 3.21(ii), for any $t>0$
\[
\frac{\mathbb{P}_\pi\left(T_D \in {\rm d}t\right)}{1-\tfrac{1}{N}}=\frac{\mathbb{P}_\rho\left(T_D > t\right)}{\mathbb{E}_\rho \left( T_D \right)}
\]
and so
\[
\mathbb{E}_\pi \left( T_D \right)
=\frac{1-\tfrac{1}{N}}{\mathbb{E}_\rho \left( T_D \right)}\int_0^\infty t \mathbb{P}_\rho\left(T_D > t\right) {\rm d}t
=\frac{1-\tfrac{1}{N}}{\mathbb{E}_\rho \left( T_D \right)}\cdot \frac{\mathbb{E}_\rho \left( T_D^2 \right)}{2}
\]
from which the result follows using $\mathbb{E}_\rho \left( T_D^2 \right) \geq \mathbb{E}_\rho \left( T_D \right)^2$.
\end{proof}

We can now prove one of the two main slow dynamic lower bounds: this one is relevant for the case $\beta+2\gamma>1$ where the stationary graph has (with high probability) a component of linear size.

\begin{proposition}\label{prop_last_singleton}
\[
\mathbb{E}\left(T_{\text{\textnormal{coal}}}\right)=\Omega\left( \frac{1}{\kappa}\log N\right)
\]
\end{proposition}

\begin{proof}
Look for singletons in the set of bounded weight
\[
W_C:=\left\{
i \in [N] : w(i)\leq C
\right\}
\]
where the induced subgraph is dominated by an Erd\H{o}s--R\'enyi graph with constant mean and so has $\Omega_{\p}(N)$ singletons at time $0$. Exploring $[N]\setminus W_C$ gives each such singleton no new edges independently with probability at least $e^{-C}$ and so we have some binomial thinning but still $\Omega_{\p}(N)$ singletons at time $0$ on the full graph.

For each of these singletons $v$ up to time $t$ on the dynamic graph, we have see the arrival of edges in the neighbourhood of $v$ with exactly Poisson distribution%
\[
\operatorname{Pois}\left( 2\kappa t w(v)\left(
1-\frac{w(v)}{w([N])}
\right) \right)
\preceq
\operatorname{Pois}\left( 2C t  \right)
\]
where the $2$ factor accounts in half for updates at $v$ and in half for updates elsewhere.

Let $Z_t$ count the number of vertices in $W_C$ that have had no neighbour throughout times $[0,t]$, and from the Poisson bound above
\[
\mathbb{E}\left(Z_{t_\epsilon}\big|Z_0\right)\geq Z_0 e^{-2\kappa {t_\epsilon} C}
=\Omega_{\p}\left(
N^{1-2C\epsilon}
\right)
\]
at time ${t_\epsilon}=\frac{\epsilon}{\kappa}\log N$. 
We also calculate the second moment
\[
\begin{split}
\mathbb{E}\left(Z^2_{t_\epsilon}\big|Z_0\right)=&
\sum_{\stackrel{i,j \in Z_0}{i \neq j}}
\exp\left(
-\kappa {t_\epsilon} \left[2w(i)\left(
1-\frac{w(i)}{w([N])}
\right)+2w(j)\left(
1-\frac{w(j)}{w([N])}
\right)-p_{ij}\right]
\right)\\
&+\sum_{i \in Z_0}
\exp\left(
-\kappa {t_\epsilon} \left[4w(i)\left(
1-\frac{w(i)}{w([N])}
\right)\right]
\right)\\
\end{split}
\]
which is the first moment squared apart from the subtraction of the overlapping edge, in particular
\[
\frac{\mathbb{E}\left(Z^2_{t_\epsilon}\big|Z_0\right)}{\mathbb{E}\left(Z_{t_\epsilon}\big|Z_0\right)^2} \leq \max_{i, j \in Z_0}e^{\kappa {t_\epsilon} p_{ij}}
\stackrel{\p}{\sim} e^{\frac{\epsilon C^2}{N}\log N}
\stackrel{\p}{\rightarrow}1
\]
and we conclude $Z_{t_\epsilon}
=\Omega_{\p}\left(
N^{1-2C\epsilon}\right)
$ by the Paley--Zygmund inequality. Hence, with high probability, there is no coalescence before any time $t_\epsilon$ with $2C\epsilon<1$. %
\end{proof}

We require the following very useful construction that relates the local limit to the graph, an adaptation of Proposition 3.1 in \cite{Norros_Reittu_2006} and proved in the same way.

\begin{proposition}\label{prop_construct_tree}
Construct the tree exploration of the neighbourhood of a vertex $v$ with the following algorithm:
\begin{enumerate}
\item \emph{explore} $v$ by giving it independently $\operatorname{Pois}(w(v))$ children (recall the weight \eqref{eq_weight});
\item\label{item_label} \emph{label} each child with an i.i.d. label in $[N] \ni i$ selected proportional to $i^{-\gamma}$;
\item \emph{explore} each vertex labelled $i$ by giving it independently $\operatorname{Pois}(w(i))$ children and then return to item \ref{item_label}.
\end{enumerate}

This (potentially infinite) tree can then be \emph{thinned} to a finite tree.
\begin{enumerate}
\setcounter{enumi}{3}
\item The root $v$ is \emph{unthinned}.
\item Iteratively, pick an arbitrary vertex adjacent to an unthinned vertex: 
if its label doesn't exist among unthinned vertices, it becomes unthinned; 
otherwise \emph{thin} the vertex by deleting it from the tree.
\item This procedure terminates with a tree containing at most $N$ unthinned vertices which is almost surely a spanning tree of the network (which spanning tree depending on the chosen order of vertices).
\end{enumerate}
\end{proposition}

To apply this exploration we will have to normalise that distribution $i^{-\gamma}$ on $i \in [N]$.

\begin{lemma}[ {\cite{jameson2003prime}}, Proposition 3.1.16 ]\label{lemma_sum}
For $\gamma \in (0,1)$
\[
\sum_{i =1}^N i^{-\gamma} = \frac{N^{1-\gamma}}{1-\gamma} + \zeta(\gamma) + O\left( N^{-\gamma} \right)
\]
where $\zeta(\gamma)<0$.
\end{lemma}

So, we have
\[
\frac{\beta }{1-\gamma}\left(
\frac{N}{i}
\right)^\gamma
\left(1
-O\left(
\frac{1}{N^{1-\gamma}}
\right)\right)\leq
w(i)
\leq \frac{\beta }{1-\gamma}\left(
\frac{N}{i}
\right)^\gamma.
\]

In particular we have the upper bound (for large $N$) by the pure power law, which is convenient in arguing the following lemma in which we have this pure power law $W$, and $W^*$ its size-biased version.

\begin{lemma}
When $\beta+2\gamma<1$, the stationary component of a walker $(X)$ has
\[
\mathbb{E}_\pi(\sC(X_t))\sim
\frac{\beta(1-2\gamma)}{(1-\gamma)^2(1-\beta-2\gamma)},
\]
in particular its expectation is bounded as $N \rightarrow \infty$.
\end{lemma}

\begin{proof}
From \cite{fernleyortgiese}, Proposition 5.9, we have a stochastic upper bound by a two-stage \emph{mixed Poisson} Galton--Watson tree of Proposition \ref{prop_construct_tree}, which is also the weak local limit of the graph around a uniform vertex. Further, using Lemma \ref{lemma_sum}, for large $N$ we can bound the weight by the pure power law and so stochastically bound a random weight by a Pareto distribution, leading to the following supertree:
\begin{itemize}
\item The general vertex has i.i.d. Pareto-distributed weight $w^*\sim W^*$ where
\[
\p(W^*\geq x)=\begin{cases}
\left(
\frac{x}{\nicefrac{\beta}{1-\gamma}}
\right)^{1-\nicefrac{1}{\gamma}}, &
x \geq \frac{\beta}{1-\gamma},\\
1, & \text{otherwise},
\end{cases}
\]
and then offspring independently $\operatorname{Pois}(w^*)$;
\item The only exception is the root, which has Pareto-distributed weight $w\sim W$ where
\[
\p(W\geq x)=\begin{cases}
\left(
\frac{x}{\nicefrac{\beta}{1-\gamma}}
\right)^{-\nicefrac{1}{\gamma}}, &
x \geq \frac{\beta}{1-\gamma},\\
1, & \text{otherwise},
\end{cases}
\]
and then offspring independently $\operatorname{Pois}(w)$.
\end{itemize}

The mean size of this tree is then the mean degree of the root multiplied by the mean size of the one-stage Galton--Watson tree from each of its children
\[
\frac{\mathbb{E}(W)}{1-\mathbb{E}(W^*)}
=
\frac{\beta}{(1-\gamma)^2}
\cdot
\frac{1}{1-\frac{\beta}{1-2\gamma}}
\]
which reduces to the claimed expression. We have argued this as an upper bound, but Lemma \ref{lemma_sum} also demonstrates a lower bound that will couple this model with high probability to the real tree for the first $o(N^{1-\gamma})$ vertices.
\end{proof}

Finally, we find a lower bound on a larger order than Proposition \ref{prop_last_singleton} which uses the small components of subcritical graph parameters.

\begin{proposition}\label{prop_coalescing_components}
When $\beta+2\gamma < 1$ meeting expects to take time at least
\[
\mathbb{E}_\pi \left( T_D \right)=
\Omega\left(
\frac{N}{\kappa}
\right)
\]
\end{proposition}

\begin{proof}
The component of a stationary walker is a tight stationary process with power law tail. So, we upper bound the meeting time of two stationary walkers by the first time they share a component.

Write $\sC_t^{(1)}$ and $\sC_t^{(2)}$ for the components of walkers $X_t^{(1)}$ and $X_t^{(2)}$ respectively. The two attach by the following mechanisms:
\begin{itemize}
\item
The walker's vertex in $\sC^{(1)}$ of weight $w$ updates at rate $\kappa$, producing if it does an exploration contained by $ D\sim \operatorname{Pois}(w)$ i.i.d. Galton--Watson explorations. Explore just in the vertices of $[N]\setminus\sC^{(2)}$ first, which is still stochastically contained by the same Galton--Watson trees of total weights $M_1,\dots,M_D$. We then find an edge to $\sC^{(2)}$ with at most probability
\[
1-\exp\left(\left(
w+\sum_{i=1}^D M_i
\right)\frac{w(\sC^{(2)})}{w([N])}
\right)
\leq
\left(
w+\sum_{i=1}^D M_i
\right)\frac{w(\sC^{(2)})}{w([N])}
\]
and so this connection is happening at rate bounded by
\[
\begin{split}
A_t&=
\kappa \frac{
w(W_t^{(1)})\left(
1+\mathbb{E}\left( M \right)
\right)w(\sC_t^{(2)})
+
w(W_t^{(2)})\left(
1+\mathbb{E}\left( M \right)
\right)w(\sC_t^{(1)})
}{w([N])}\\
&\leq
\frac{
2\kappa\left(
1+\mathbb{E}\left( M \right)
\right)
w(\sC_t^{(1)})w(\sC_t^{(2)})
}{w([N])}.
\end{split}
\]
\item At rate bounded by
\[
B_t=
\kappa \frac{w(\sC^{(1)})w(\sC^{(2)})}{w([N])}
\]
a vertex outside of both components updates and connects to both, so attaching them.
\end{itemize}

On the event that $\sC_x^{(2)} \neq \sC_x^{(1)}$, we can generate $\sC_x^{(2)}$ as an i.i.d. copy of $\sC_x^{(1)}$. Hence
\[
\mathbb{E}\left(
B_x ; \sC_x^{(1)} \neq \sC_x^{(2)}
\right)
\leq
\kappa \frac{\mathbb{E}\left(w(\sC^{(1)})\right)^2}{w([N])}
\]
and similarly for $A_x$. So, the clock of interest 
\[
C_t=\int_0^t (A_x+B_x) \mathbbm{1}_{\sC_x^{(1)} \neq \sC_x^{(2)}} {\rm d}x
\]
has
\[
\mathbb{E}(C_t)
=\int_0^t \mathbb{E}(A_x+B_x; \sC_x^{(1)} \neq \sC_x^{(2)}) {\rm d}x
\leq\kappa t\frac{\left(\mathbb{E}(w(\sC))\right)^2}{w([N])}\left(
3+2\mathbb{E}\left( M \right)
\right)=\Theta\left(
\frac{\kappa t}{N}
\right).
\]

By upper bounding the rate of arrival we are constructing an exponential variable as a stochastic lower bound. Hence, with probability at least $\nicefrac{1}{e}$, there will be no meeting while $C_t<1$. Simply by Markov's inequality we have $\p(C_t>1)\leq \mathbb{E}(C_t)\leq \nicefrac{1}{2e}$ for some $t$ on the order $\nicefrac{N}{\kappa}$ and so, at this time $t$, we have not met with probability at least the difference $\nicefrac{1}{2e}$.
\end{proof}

\section{Upper bounds}

To upper bound consensus, as was noted in Lemma \ref{lemma_coal_cons}, we will have to consider not just meeting but full coalescence of $N$ walkers. The arguments change depending on whether the graph parameters are above the critical line so that we use the giant component to find other walkers, or below where small components must connect directly to each other.

\subsection{Without giant component}

The slightly simpler of the two cases is when $\beta+2\gamma<1$ which we consider in this section. First we collect some relatively standard definitions of mixing quantities.

\begin{definition}\label{def_mixing}
For a Markov chain on state space $S$ we have mixing distances
\[
d(t):=\frac{1}{2}\max_{x \in S}\| p^{(t)}_{x,\cdot}-\pi(\cdot) \|_1,
\]
\[
\bar{d}(t):=\frac{1}{2}\max_{x,y \in S}\| p^{(t)}_{x,\cdot}-p^{(t)}_{y,\cdot} \|_1,
\]
\[
s(t)=:\max_{y \in S}\left( 1-\min_{j \in S} \frac{p^{(t)}_{x, y}}{\pi(j)}\right),
\]
the mixing time $t_{\rm mix}$ is then defined as 
\[
t_{\text{\textnormal{mix}}}:=\min \Big\{ t\geq 0 : d(t)\leq \nicefrac{1}{e}\Big\},
\]
and by writing $Q$ for the generator of the chain we define relaxation time
\[ 
t_{\text{\textnormal{rel}}}:=\max \left\{ \frac{1}{\lambda} : \lambda \text{ a positive eigenvalue of } -Q \right\} \]
where this final definition also requires that the Markov chain is reversible and hence has a real spectrum.
\end{definition}

The threshold $\nicefrac{1}{e}$ is somewhat arbitrary but in Chapter 4 of \cite{aldous-fill-2014} we see that this definition provides $t_{\rm rel} \leq t_{\rm mix}$ for any reversible Markov chain.

\begin{proposition}\label{prop_graph_mixing}
Mixing of $G_\kappa$ is $O\left( \frac{1}{\kappa}\log N\right)$.
\end{proposition}

\begin{proof}
From any initial graph, after an update at $v \in [N]$ we have the edges incident to $v$ at stationarity. Vertex updates arrive at rate $\kappa N$ with Poisson process concentration, and the coupon collector problem tells us that the first time every vertex has updated is after $(1+o_{\p}(1))N \log N$ updates. So, we have constructed a strong stationary time for the graph at time \[(1+o_{\p}(1))\frac{(N \log N)}{\kappa N}.\]
\end{proof}

Graph mixing is now understood but we want to build this into mixing of the chain $M$ with attached walkers \eqref{eq_M} on the same timescale. First, we need a technical lemma on the stationary degree histories.

\begin{lemma}\label{lem_stat_degree}
For any $v \in [N]$ and $C\geq 5$
\[
\p_\cL\left(
\frac{1}{t}
\int_{s=0}^t
\de_s(v)
{\rm d}s
>(C+1) w(v)
\right)
\leq
4\exp
\left(
-\frac{Cw(v)}{5}
(\kappa t \vee 1)
\right).
\]
\end{lemma}

\begin{proof}
First if $t \leq \frac{1}{\kappa}$ we use the Chernoff bound
\[
\mathbb{P}\left(
\de_0(v)>\frac{C+1}{2} w(v)
\right)
\leq e^{-\frac{C+1}{2} w(v)}
e^{(e-1) w(v)}
\]
and note that over a period of length at most $t$ we expect single vertices to attach (by an update elsewhere) numbering
\[
\kappa t \sum_{w \neq v} p_{v w} \leq w(v)
\]
which we bound with the same Chernoff bound. Given both Chernoff events, the maximum degree before the first update of $v$ is bounded by $(C+1) w(v)$ and so too must be the mean degree of interest.

When the vertex $v$ updates we iterate the argument. We are double-counting the time interval, but loosely we can say that we have $1+\operatorname{Pois}(\kappa t)\preceq 1+ \operatorname{Pois}(1)$ attempts. By Markov's inequality, the probability that one of them fails one of the two Chernoff bounds is bounded by twice the probability that the first one fails, yielding
\[
2\left(1-\left(
1-
e^{-\frac{C+1}{2} w(v)}
e^{(e-1) w(v)}
\right)^2\right)
\leq
4
e^{-\frac{C+1}{2} w(v)}
e^{(e-1) w(v)}
\leq
4
e^{-\frac{C}{4} w(v)}
\]
using in the final step that $C\geq 5$.

For larger $t$, the neighbourhood of $v$ has relaxation time
\[
\frac{1}{\lambda_1}=:
t_{\rm rel} \leq t_{\rm mix}\left(\frac{1}{e}\right)\leq \frac{1}{\kappa}
\]
by considering only the updates at $v$, and we use this to apply \cite{lezaud01}, Remark 1.2, with:
\[
b^2=w(v)(1+w(v)),
\qquad
a=N,
\]
\[
\gamma=Cw(v) \gg \frac{w(v)(1+w(v))}{N} = \frac{b^2}{a}
\]
which is then the result (recall $\cL(\de(v))<w(v)$)
\[
\p_\cL\left(
\int_{s=0}^t
\left(
\de_s(v)
-\cL(\de(v))
\right)
{\rm d}s
>C t w(v)
\right)
\leq
\exp
\left(
-(1+o(1))\frac{Cw(v)\kappa t}{4}
\right)
\]
with the limit being $N\rightarrow\infty$ and hence the other case of the claim for large $N$.
\end{proof}

Keeping degrees close to their weights in this sense facilitates the following result. The idea here is to formalise that a walker is mixed after seeing an update at the vertex it occupies: at least, with positive probability.

\begin{lemma}\label{lem_sep}
After any update at the vertex of a walker, it is stationary after a further unit time with probability $\nu_{\beta,\gamma}=\Omega(1)$.
\end{lemma}

\begin{proof}
Suppose the update occurs at $i\in [N]$ and by timeshift w.l.o.g. at time $0$. Consider the interval
\[
t_i:=\frac{1}{w(i) \vee 1},
\]
divide this time into two periods
\[
\left(
t_i-t_j
\right)^+
\quad
\text{then}
\quad
t_i \wedge t_j,
\]
and control that:
\begin{enumerate}%
\item the walker doesn't move in the first period and then moves nowhere other than $j$ in the second;
\item the walker attempts to move to $j$ in the second period and discovers $i \sim j$ when this attempt is made;
\item there are no further moves by the walker during the second period.
\end{enumerate}

The first and third points are controlled first by Lemma \ref{lem_stat_degree} at $i$ over both periods and at $j$ for just the second, from which for either $k=i$ or $k=j$
\[
\p_\cL\left(
\int_{s=0}^{t_k}
\de_s(k)
{\rm d}s
>C+1
\right)
\leq
\exp
\left(
-\frac{Cw(k)}{5}
\right)
\]
and so the probability of no step from $i$ in time $t_i$ is at least
\[
\left(
1-4e^{
-\nicefrac{w(i)}{5}
}
\right)
e^{-C-1}
\geq
\frac{C w(N)-10}{C w(N)-5}e^{-C-1}
\]
by taking $C$ sufficiently large, for example $C\geq \nicefrac{11}{\beta}$. From $j$, we have an extra edge but there we still have probability at least
\[
\frac{C w(N)-10}{C w(N)-5}
e^{-C-1-t_j}
>\frac{C w(N)-10}{C w(N)-5}
e^{-C-1}\frac{w(j)}{1 +w(j)}
.
\]

For (ii), the walker attempts to move to $j$ in the second period with probability
\[
1-e^{-t_i \wedge t_j}\geq\frac{1}{2+w(i)\vee w(j)}.
\]

When it does, we have edges between $i$ and $j$ at time $0$ (the simple graph version of) Poisson with mean $\nicefrac{w(i)w(j)}{w([N])}$ and so the edge is present with probability
\[
p_{ij}
\geq
\frac{w(i)w(j)}{w([N])+w(i)w(j)}.
\]

Overall we satisfy all three conditions with minimal probability at least
\[
\begin{split}
\Bigg(
\frac{C w(N)-10}{C w(N)-5} &e^{-C-1}
\Bigg)^2
\left(
\frac{ 1}{ 2+w(i)\vee w(j)}
\cdot
\frac{w(i)w(j)}{w([N])+w(i)w(j)}
\cdot
\frac{w(j)}{ 1 +w(j)}
\right)\\
&=
\left(
\frac{C w(N)-10}{C w(N)-5}e^{-C-1}
\right)^2
\left(
\frac{ w(i)\vee w(j)}{ 2+w(i)\vee w(j)}
\cdot
\frac{w(i) \wedge w(j)}{w([N])+w(i)w(j)}
\cdot
\frac{w(j)}{ 1 +w(j)}
\right)\\
&\geq
\left(
\frac{C w(N)-10}{C w(N)-5}e^{-C-1}
\right)^2
\left(
\frac{w(N)}{2+w(N)}
\right)
\left(
\frac{w(N)}{1+w(N)}
\right)
\left(
\frac{w(i) \wedge w(j)}{w([N])+w(i)w(j)}
\right)\\
&\geq
e^{-2C-1}
\left(
\frac{w(N)}{2+w(N)}
\right)^2
\left(
\frac{w(N)}{w([N])+w(N)w(1)}
\right)=\Omega\left( \frac{1}{N} \right). \\
\end{split}
\]
Note in the fourth line we require $C$ large. Then, multiply this by an $N-1$ factor for the claimed constant probability.

At $i$, we simply have a $\left(
\frac{C w(N)-10}{C w(N)-5}e^{-C-1}
\right)=\Omega(1)$ probability that it never moved and thus we have achieved a constant separation over all of $[N]$ at time $t_i$, i.e. a strong stationary time with positive probability. This is then preserved to time $1\geq t_i$ as claimed.
\end{proof}

This construction immediately allows us to extend Proposition \ref{prop_graph_mixing} to a result on the mixing of the chain \eqref{eq_M} with walkers.

\begin{corollary}\label{corr_mixing}
Mixing of the two-walker chain $M$ is $O\left(1 + \frac{1}{\kappa }\log N\right)$.
\end{corollary}

\begin{proof}
We rely on Lemma \ref{lem_sep}. 
For $C$ large enough, after time 
\[
3 \cdot \left( C + \frac{1}{\kappa} \right) \cdot 
\frac{1}{\nu_{\beta,\gamma}}=O\left(1 + \frac{1}{\kappa }\right)
\]
we have a small probability to have not seen strong stationary times for both walkers.

Once both walkers are independently stationary, this is preserved regardless of the graph and so we can wait the additional $O(\tfrac{1}{\kappa}\log N)$ time of Proposition \ref{prop_graph_mixing} for the graph to also mix.
\end{proof}

From just sampling stationary positions with this lemma we have good control on the coalescence time of $\nicefrac{1}{\kappa}$ walkers in the subcritical graph, which is the full upper bound of the subcritical case of Theorem \ref{theorem_dynamic_cons} when $\kappa\leq \nicefrac{1}{N}$. To use uniforms for coalescence, we recall the birthday problem result.

\begin{lemma}[Birthday problem]\label{lem_birthday}
Let $X$ be the number of multiply occupied sites by $k$ i.i.d. uniform labels in $[N]$. We find when $k\leq \sqrt{N}$
\[
\p\left(
X =0
\right) =
\exp
\left(
-
\frac{k^2}{2N}
+
O\left(
\frac{k}{N}
\right)
\right)
.
\]
\end{lemma}

For the other regime where we expect a large number of coalescences, we have a brief second moment argument.

\begin{lemma}\label{lem_advanced_birthday}
Let $X$ be the number of multiply occupied sites by $k$ i.i.d. uniform labels in $[N]$. We find when $N \geq k \geq  \sqrt{N}$
\[
\p\left(
X \geq \frac{k^2}{6N}
\right) \geq
\frac{1}{73}
.
\]
\end{lemma}

\begin{proof}
We stochastically lower bound $X \succeq Y$, where $Y$ represents the number of sites occupied by exactly $2$ vertices. For $u \in [N]$
\[
\p(u\in Y)
=
\binom{k}{2}\frac{1}{N^2}
\left(
1-\frac{1}{N}
\right)^{k-2}
\]
and so applying standard exponential inequalities
\begin{equation}\label{eq_exponential_ineq}
\frac{1}{N}\binom{k}{2}
e^{-\frac{k}{N-1}}
\leq
\mathbb{E}(Y)
\leq
\frac{1}{N-2}\binom{k}{2}
e^{-\frac{k}{N}}.
\end{equation}

For the second moment, take $u \neq v \in [N]$ where
\[
\p(u,v\in Y)
=
\binom{4}{2}
\cdot
\binom{k}{4}
\left(
\frac{2}{N}
\right)^4
\left(
1-\frac{2}{N}
\right)^{k-4}
\leq
\frac{4k^2(k-1)^2}{N^4}
e^{-\frac{2k-8}{N}}
\]
hence we upper bound
\[
\begin{split}
\mathbb{E}(Y^2)
&\leq
N\p(u\in Y)
+N^2\p(u,v\in Y)\\
&\leq 
\frac{1}{N-2}\binom{k}{2}
e^{-\frac{k}{N}}
+
\frac{4k^2(k-1)^2}{N^2}
e^{-\frac{2k-8}{N}}.
\end{split}
\]

We conclude by the Paley--Zygmund inequality
\[
\begin{split}
\p\left(
Y\geq \frac{\mathbb{E}(Y)}{2}
\right)
&\geq \frac{1}{4}
\cdot
\frac{\frac{1}{N^2}\binom{k}{2}^2
e^{-\frac{2k}{N-1}}}{\frac{1}{N-2}\binom{k}{2}
e^{-\frac{k}{N}}
+
\frac{4k^2(k-1)^2}{N^2}
e^{-\frac{2k-8}{N}}}\\
&=
\frac{1}{4}
\left(
\frac{N^2}{N-2} \frac{2}{k(k-1)}e^{\frac{k}{N}\frac{N+1}{N-1}}
+16e^{\frac{2k}{N-1}-\frac{2k-8}{N}}
\right)^{-1}\\
&\geq
\left(
8+64
+O\left(
\frac{1}{\sqrt{N}}
\right)
\right)^{-1}.
\end{split}
\]
which is the claimed result using \eqref{eq_exponential_ineq} because
\[
\frac{\mathbb{E}(Y)}{2}
\geq
\frac{1}{2N}\binom{k}{2}
e^{-\frac{N}{N-1}}
\geq
\frac{k(k-1)}{2Ne}\left(
1-\frac{1}{N-1}
\right)
\geq
\frac{k^2}{6N}.
\]
\end{proof}

These two regimes of the birthday problem combine to produce essentially the Kingman coalescent \cite{kingman} and the following order of coalescence time.

\begin{proposition}\label{prop_subcrit_coal}
From arbitrary initial positions and initial graph
\[
\mathbb{E}\left(
T_{\rm coal}
\right)
=
O\left(
\frac{N}{\kappa}
+N \log N
\right),
\]
\end{proposition}

\begin{proof}

We upper bound coalescence by constructing a certain subset of meetings which we can control, and all other meetings not described are simply ignored for the construction of the upper bound. 

Suppose we have $w$ walkers, and allow time $\nicefrac{1}{\kappa}$ in which to see an update at each walker's site with probability $1-\nicefrac{1}{e}$. If a walker does see an update, pause the $\nicefrac{1}{\kappa}$ timers and ignore other updates for the time $1$ of Lemma \ref{lem_sep}. Giving each walker at most $1$ update in this way, we find a subset of the walkers numbering
\[
B \sim 
\operatorname{Bin}
\left(
w,
\nu_{\beta,\gamma}
\right)
\]
which are independently stationary at time $w+\nicefrac{1}{\kappa}$. When $w<\sqrt{N}$, we apply Lemma \ref{lem_birthday}
\[
\p\left(
X \geq 1
\right)
\sim\frac{k^2}{2N}
\]
and immediately from this geometric variable we have time $O(\nicefrac{N}{w^2})$ between coalescences.

For $w \geq \sqrt{N}$, the geometric of Lemma \ref{lem_advanced_birthday} produces an upper bound on the coalescence process taking steps of size $\Omega(\nicefrac{w^2}{N})$, which can be thought of as the same expression for the \emph{rate} of consensus and thus altogether produces the consensus order
\[
\sum_{w=2}^N O\left( \frac{N}{w^2} \right) \cdot
\left(
w+\frac{1}{\kappa}
\right)
=O\left(
\frac{N}{\kappa}
+N \log N
\right).
\]
\end{proof}

\subsection{With giant component}

It remains to consider the cases where $\beta+2\gamma>1$ and so we have a dynamic component of $\Omega_{\p}(N)$ vertices which will facilitate coalescence. First we recall the electrical network which is very powerful for understanding the return probabilities of a reversible chain, for a more detailed treatment see \cite{wilmer2009markov}, Chapter 9.

\begin{definition}[Electrical network]\label{def_electric}
A reversible continuous-time Markov chain corresponds to the electrical network with conductance $c(x,y)=\pi(x)q_{xy}$ between every $x \neq y$ in its state space. 
\end{definition}

To control returning to a subset with the electrical network, we have to short that subset (that is, connect all internal edges with $0$ resistance edges). For the Markov chain, this produces a \emph{collapsing} in the sense of \cite{aldous-fill-2014}, Section 2.7.3, turning the subset into a single state with rates averaged according to $\pi$.

\begin{lemma}\label{lemma_shorting}
The collapsed chain corresponds to the shorted electrical network.
\end{lemma}

\begin{proof}
From a generator $Q=[q(x,y)]_{x=1 \, y=1}^{N \,  \, \,   \,  \,  \, N}$ we collapse the space $A$ to an element $a$ to obtain collapsed rates
\[
\begin{split}
\forall y \notin A : \,
\tilde{q}(a,y)
&=
\sum_{x \in A}\frac{\pi(x)q(x,y)}{\pi(A)},\\
\forall x \notin A : \,
\tilde{q}(x,a)
&=
\sum_{y \in A}q(x,y),\\
\forall x,y \notin A : \,
\tilde{q}(x,y)
&=
q(x,y).
\end{split}
\]

This can be quickly seen to be reversible, with stationary distribution $\tilde{\pi}(a)=\pi(A)$ and so the conductances of Definition \ref{def_electric} become
\[
\forall y \notin A : \,
\tilde{c}(a,y)=\sum_{x \in A}\pi(x)q(x,y)=\sum_{x \in A}c(x,y)
\]
thus the network $\tilde{c}$ is indeed the shorted version of $c$.
\end{proof}

We are now ready to prove a useful result on the escape probability. Here  $T^{(1)}_{\rm hit}(v)$ denotes the hitting time for walker $1$, i.e. for $M$ the first time that the first coordinate is $v$.  Recall that meeting is hitting of the the diagonal $D:=\{(i,i):i\in [N]\}\subset [N] \times [N]$.

\begin{proposition}\label{prop_escape}
When $\beta+2\gamma>1$, we find for the two-walker chain
\[
\inf_{v \in [N]}
\p_\rho\left(
T_D
>
T^{(1)}_{\rm hit}(v)
\right)
=\Omega\left(\frac{1}{\log N}\right).
\]
\end{proposition}

\begin{proof}
First note that
\[
\p_\rho\left(
T^{(1)}_{\rm hit}(v)=0 \, \bigg| \, \rho^-\in
\{(v,v)\}\times \sG_N
\right)=\frac{1}{2}
\]
and so the claim holds on this event.%

For the complement event, in the chain $M_t$ of \eqref{eq_M} we collapse the source set
\[
\blacklozenge=
\left\{
\left(
x,x,g
\right)
:
x \in [N] \setminus \{v\}, \, 
g \in \sG_N
\right\}
\]
and separately collapse the sink set
\[
\lozenge=
\left\{
\left(
v,y,g
\right)
:
y \in [N], \, 
g \in \sG_N
\right\}.
\]

Collapsing, as we saw in Lemma \ref{lemma_shorting}, corresponds to shorting the electrical network. 
Both source and sink then have stationary measure $\nicefrac{1}{N}$, moreover any graph $g$ containing an edge $\{x,y\} \not\ni v$ receives rate
\[
\widetilde{q}\left(
\blacklozenge,
(x,y,g)
\right)
=
\cL(g)
\frac{\nicefrac{2}{N^2}}{\nicefrac{1}{N}}
\]
which is constant in the edge. So, the exit distribution of $\blacklozenge$ selects a stationary graph biased with its total degree and then a uniform edge in that graph (with half probability if the edge contains $v$)---this is exactly the distribution $\rho$ conditioned on $\{\rho^-\notin\{
(v,v)\}\times \sG_N \}$ which is our interest.

On this shorted electrical network, we construct a unit flow from source to sink by which to control the resistance.

For some large enough $C$, let $G_{x}$ be the set of good graphs with the following properties. Note by Lemma \ref{lem_asymptotic_equiv} we only need to argue that they hold with positive probability for the stationary graph and then they will for $\rho$ by asymptotic equivalence.

\begin{enumerate}
\item\label{item_connected} $x,v \in [N]$ are connected. %
\item\label{item_removing_path} After removing the shortest path between $x$ and $v$, $x$ remains in a component for which the breadth-first exploration tree has at least $\nicefrac{N}{C}$ leaves. %
\item\label{item_diameter} The diameter of their component is at most $C \log N$. 
\item\label{item_edges} $g$ has at most $CN$ edges, .
\end{enumerate}

Property \ref{item_edges} holds with high probability, from the simple Chernoff bound as in the proof of Lemma \ref{lem_asymptotic_equiv}.

By the diameter result \cite{fernley2022discursive}, Theorem 2.3, we know Property \ref{item_diameter} occurs with high probability uniformly in $x$, for some $C=\Omega(1)$ large enough. 

For properties \ref{item_connected} and \ref{item_removing_path}, we recall the tree exploration of Proposition \ref{prop_construct_tree}. Construct the tree from $x$, where
\[
\p\left(
\de(x)\geq 2
\right)
\geq 1-e^{-\beta}(1+\beta).
\]

If we do see $\{\de(x)\geq 2\}$, pick an arbitrary two children. Because their offspring mean is larger than $1$, the Galton--Watson trees of both these children survive with approximately the positive probability $\rho^2$ of \cite{bollobas07}, Theorem 6.4.

Following the thinning procedure in the breadth-first order in both Galton--Watson trees, we can check say $\sqrt[3]{N}$ vertices in both these branches and with high probability not see any repeated labels: they are all unthinned.

Construct also an exploration from $v$ where $\p\left(
\de(v)\geq 1
\right)
\geq 1-e^{-\beta}$ and then conditionally that Galton--Watson tree survives with probability $\rho$. Thinning this Galton--Watson tree in the breadth-first order also leaves $\sqrt[3]{N}$ labels which are with high probability distinct from those in the previous label sets. Completing these explorations in any order, with high probability both $x$ and $v$ must be in the giant component and hence Property \ref{item_connected}.

Completing the exploration from the root $x$ in the normal breadth-first order, the first appearance of the label $v$ is in at most one of the two branches previously explored. Hence after deleting that path to $v$ we still have $\sqrt[3]{N}$ unthinned vertices and so the completed exploration will still with high probability form a spanning tree of the giant. We conclude Property \ref{item_removing_path} because we separately say the giant has $\Omega_{\p}(N)$ leaves with high probability.

The first two can happen together with positive probability on the same graph exploration and the last two are high probability conditions. So, we have 
\[\min_{x}\cL(G_{x})
\geq \rho^3 \left(
1-e^{-\beta}(1+\beta)
\right)\left(
1-e^{-\beta}
\right)-o(1)
=\Omega(1)\]
as $N\rightarrow\infty$.

For each $g \in G_{x}$, remove a shortest path $P_{g,x}$ between $x$ and $v$ and construct a spanning tree $T_{g,x}$ of the remaining giant component, where $T_{g,x}$ has root $x$ and has minimal depth (by for example a breadth-first exploration from $x$). Define a unit flow $w_{g,x}$ from $x$ to the leaves $L_{g,x}$ of $T_{g,x}$ ($w_{g,x}$ is a function on the vertices measuring the total flow into the vertex, with $w_{g,x}(x)=1$) such that $w_{g,x}$ is also \emph{uniform} on the leaves
\[
w_{g,x}\Big|_{L_{g,x}}\equiv \frac{1}{|L_{g,x}|}\leq\frac{C}{N}.
\]

Write $P_{g,x}=(a_0,a_1,\dots)$, $b_{\ell}$ for a vertex of depth $\ell$ in the tree and $b_{\ell+1}$ for a child of $b_{\ell}$. Whenever $b_\ell \notin L_{g,x}$ and so this child exists, define
\[
\theta_{g,x}\left(
(x, b_{\ell},g) \rightarrow (x, b_{\ell+1},g)
\right)
=
\frac{1}{N}
\cdot
\frac{\cL(g)
}{\cL(G_{x})}
\cdot
w_{g,x}(b_{\ell+1}))
\]
and upon hitting a leaf $b \in L_{g,x}$, the flow goes directly to the sink
\[
\theta_{g,x}\left(
(a_k, b,g) \rightarrow (a_{k+1}, b,g)
\right)
=
\frac{1}{N}
\cdot
\frac{\cL(g)
}{\cL(G_{x})}
\cdot
\frac{1}{|L_{g,x}|}
.
\]
Note that we have constructed a unit flow: dividing the flow by source vertex, then distributing over the set of good graphs with the conditional stationary distribution, and finally by requiring that $w$ is a unit flow. 
It remains to calculate the energy
\[
\cE(\theta)=\sum_e \theta(e)^2 r_e \qquad  \text{ for }\theta=\sum_{x = 1}^N\sum_{g \in G_x} \theta_{g,x}.
\]

Every edge of a second coordinate move has the source $x$ and graph $g$ attached: when defined on disjoint edges, the energy of the sum flow is simply the sum energy. Decomposing $\cE(\theta)=\cE^{(1)}(\theta)+\cE^{(2)}(\theta)$ into the sum over first coordinate and second coordinate edges respectively, this means
\[
\begin{split}
\cE^{(2)}(\theta)&=
\sum_{x=1}^N
\sum_{g \in G_x}
\sum_{b \in T_{g,x}\setminus \{x\}}
\frac{N^2}{\cL(g)}
\cdot
\left(
\frac{\cL(g)
}{N\cL(G_{x})}
w_{g,x}(b))
\right)^2\\
&\leq
\sum_{x=1}^N
\sum_{g \in G_x}
\frac{\cL(g)
}{\cL(G_{x})^2}
\sum_{b \in T_{g,x}\setminus \{x\}}
w_{g,x}(b)\\
&\leq
\frac{C \log N}{\min_{x}\cL(G_{x})^2}
\sum_{x=1}^N
\sum_{g \in G_x}
\cL(g)
=
\frac{C N \log N}{\min_{x}\cL(G_{x})^2}
\end{split}
\]
using that every walker edge has resistance $\nicefrac{N^2}{\cL(g)}$.

For a first coordinate move, we only need to total flow for the same $g$ and hit leaf $b$ but nevertheless we have $g\in G_y$ and $b \in V(g)$ for many other $y \in [N]$. That is to say, any edge $e \in E(g)$ cannot receive more flow as first coordinate (with second coordinate any particular leaf, and third coordinate $g$) than
\[
\sum_{y=1}^N
\mathbbm{1}_{g \in G_y}
\mathbbm{1}_{e \in P_{y,g}}
\frac{\cL(g)
}{N\cL(G_{y})|L_{g,y}|}
\leq
\frac{\cL(g)}{\min_{x}\cL(G_{x})}
\cdot
\frac{C\log N}{\tfrac{N}{C}-C\log N}
\sum_{y=1}^N
\mathbbm{1}_{g \in G_y}
\mathbbm{1}_{e \in P_{y,g}}
.
\]

From this we deduce, by again adding terms of resistance $\nicefrac{N^2}{\cL(g)}$,
\[
\begin{split}
\cE^{(1)}(\theta)
&\leq \frac{1}{\min_{x}\cL(G_{x})^2}
\sum_{g}
\sum_{b=1}^N
\sum_{e \in  E(g)}
 \cL(g)
\frac{C^2 }{N^2}
\left(
\sum_{y=1}^N
\mathbbm{1}_{g \in G_y}
\mathbbm{1}_{e \in P_{y,g}}
\right)^2\\
&\leq \frac{1}{\min_{x}\cL(G_{x})^2}
\sum_{g \in \bigcup_y G_y}
N
 \cL(g)
\frac{C^2 }{N^2}
\sum_{e \in  E(g)}
\sum_{y=1}^N
\sum_{z=1}^N
\mathbbm{1}_{e \in P_{y,g} \cap P_{z,g}}
\\
&\leq \frac{1}{\min_{x}\cL(G_{x})^2}
\sum_{g \in \bigcup_y G_y}
 \cL(g)
\frac{C^2 }{N}
\sum_{y=1}^N
\sum_{z=1}^N
\left|
P_{y,g} \cap P_{z,g}
\right|
\\
&\leq \frac{1}{\min_{x}\cL(G_{x})^2}
\sum_{g \in \bigcup_y G_y}
 \cL(g)
\frac{C^2 }{N}
\cdot N^2 \cdot C \log N
\leq \frac{C^3 N \log N}{\min_{x}\cL(G_{x})^2}.
\end{split}
\]

Note that the source set $\blacklozenge$ is only left by walker moves with no dependence on $\kappa$, and so we upper bound the conductance of the source
\[
c(\blacklozenge)\leq\sum_{g \in \sG_N}
\sum_{x=1}^N
\frac{\cL(g)}{N^2} \de_g(x)
=
O\left(
\frac{1}{N}
\right).
\]

We conclude by Proposition 9.10 of \cite{wilmer2009markov} that this energy bounds the resistance and then by Theorem 9.5 of \cite{wilmer2009markov} that the escape probability is at least
\[
\frac{1}{c(\blacklozenge)\left(\cE^{(1)}(\theta)+\cE^{(2)}(\theta)\right)}
\geq
\frac{\min_{x}\cL(G_{x})^2} {O(\nicefrac{1}{N}) N  \log N}
\]
uniformly in the target $v$, as claimed.
\end{proof}

The previous result on hitting without meeting we can manipulate into a result on \emph{mixing before meeting}. Recalling from Lemma \ref{prop_ergodic_return} that the expected meeting time from $\rho$ is always $\Theta(N)$; this will allow us to relate meeting from $\rho$ to the more difficult meeting from $\pi$.

Before that argument, we have the following technical lemma on the maximum degree in a vertex history.

\begin{lemma}\label{lem_max_deg}
For the dynamic graph, we find some universal constant $C$ and every $v \in [N]$
\[
\p_\rho\left(
\sup_{t<N^2}
\de_t(v)>C w(v)\log N
\right)
\leq
\frac{C w(v)}{N}
\]
\end{lemma}

\begin{proof}
First we note that because a graph $g$ with degree $\de^{(g)}(v)=k$ has size-biased mass (see the proof of Lemma \ref{lem_asymptotic_equiv})
\[
\cL^*(g)=\frac{2k+\de^{(g\setminus\{k\})}([N]\setminus\{k\})}{\cL(\de^{(g)}([N]))}\cL(g)
\]
the initial neighbourhood on a size-biased graph has mass
\[
\cL^*(\de^{(g)}(v)=k)=
\frac{2k-2\cL(\de^{(g)}(v))+\cL(\de^{(g)}([N]))}{\cL(\de^{(g)}([N]))}\cL(\de^{(g)}(v)=k)
\]
and so, by calculating the total variation distance, we can couple the two neighbourhoods with probability bounded by
\[
\frac{2\cL(\de^{(g)}(v))}{\cL(\de^{(g)}([N]))}=O\left(
\frac{w(v)}{N}
\right)
\]
where the constant factor is uniform also in $v$.

If they are coupled, the degree process at $v$ is a biased stationary walk $(X_t)_t$ on the $(N-1)$-cube which redraws each individual edge at rate $\kappa$, and also jumps to an independent stationary position (or redraws all edges at once) at rate $\kappa$. Write $\pi$ on $\{0,1\}^{N-1}$ for the stationary measure of this process. We consider the subset $S=\left\{\sum \de(X) > C w(v)\log N\right\}$, and from the Chernoff bound
\[
\begin{split}
\pi\left(
S
\right)
&\leq
\p\left(
\operatorname{Pois}\left( w(v)\right) > C w(v)\log N
\right)
\leq
e^{ w(v)
(e-1)- C w(v)\log N
}\\
&\leq
e^{ \tfrac{\beta}{1-\gamma}\left(
e - C\log N
\right)
}
=N^{-\Omega(C)}.\\
\end{split}
\]

Collapsing the set $S$ produces a state with rate
\[
q(S)\leq \kappa\left(1+ \frac{C w(v)\log N}{w(N)}\right)
\]
and so for the hitting probability, by \cite{aldous-fill-2014}, Lemma 3.17 and Proposition 3.23,
\[
\begin{split}
\p_\pi\left(
\sup_{s<t} \de\left(X_s\right) >  C w(v)\log N
\right)
&=\p_\pi\left(
T_S<t
\right)
\leq \frac{t_{\rm rel}}{\mathbb{E}_\pi\left(
T_S
\right)}
+\exp\left(
-\frac{t}{\mathbb{E}_\pi\left(
T_S
\right)}
\right)\\
&=O\left(
(t+t_{\rm rel}) q(S)\pi(S)
\right)\\
&=
(\kappa t+\log N) (w(v)\log N)N^{-\Omega(C)}
\end{split}
\]
using Corollary \ref{corr_mixing} to bound $t_{\rm rel}$. Because $w(v)\leq N$, taking $C$ sufficiently large then makes this a smaller order than the initial probability to not couple to stationarity.
\end{proof}

Controlling degrees up to a polylogarithmic factor allows us to understand meeting up to a polylogarithmic factor in the following sense.

\begin{corollary}\label{corr_escape_prob}
For the two-walker chain
\[
\p_\rho\left(
T_D
>
\frac{N}{\log^5 N}
\right)
=\Omega\left(\frac{1}{\log N}\right).
\]
\end{corollary}

\begin{proof}
Suppose for contradiction that instead
\[
\p_\rho\left(
T_D
>
\frac{N}{\log^5 N}
\right)
=o\left(\frac{1}{\log N}\right).
\]

From Proposition \ref{prop_escape}
\[
\p_\rho\left(
T^{(1)}_{\rm hit}(v)
<
\frac{N}{\log^5 N}
\right)
\geq
\p_\rho\left(
T_D
>
T^{(1)}_{\rm hit}(v)
\right)
-
\p_\rho\left(
T_D
>
\frac{N}{\log^5 N}
\right)
=\Omega\left(\frac{1}{\log N}\right)
\]
and so by time $\nicefrac{N}{\log^5 N}$ we expect to have covered
$
\Omega\left(\nicefrac{N}{\log N}\right)
$
vertices. 
By the Paley--Zygmund inequality with some small $\epsilon>0$, this exceeds $\nicefrac{\epsilon N}{\log N}$ with probability $
\Omega\left(\nicefrac{1}{\log^2 N}\right)
$. 
Necessarily, this includes some set $V$ of
\[
\Omega\left(\frac{N}{\log N}\right)
\text{ vertices of weight }
O\left(\log^\gamma N\right).
\]

By Lemma \ref{lem_max_deg}, any such vertex has probability $O\left(
\nicefrac{\log^\gamma N}{N}
\right)$ to ever see degree larger than $O\left(
\log^{1+\gamma} N
\right)$. Hence by Markov's inequality, with high probability at least half of $V$ satisfies the degree bound and so have expected holding time $\Omega\left(
\log^{-1-\gamma} N
\right)$.

The total expected time is then
\[
\Omega\left(\frac{N}{\log^{4+\gamma} N}\right)
\]
which contradicts that we are considering cover at time $\nicefrac{N}{\log^5 N}$.
\end{proof}

Now the point of this sequence of arguments, as was briefly mentioned before, is to turn the lower bound for meeting from $\rho$ into an upper bound for meeting from $\pi$.

\begin{corollary}\label{cor_large_k_meeting}
For the two-walker chain, when $\beta > 1-2\gamma$ and $\kappa\geq \frac{\log^7 N}{N}$
\[
\mathbb{E}_\pi\left(T_D\right)=
    O\left(  N \log N \right).
\]
\end{corollary}

\begin{proof}
Because $\kappa\geq \nicefrac{(\log^7 N)}{N}$, by Corollary \ref{corr_mixing} the system $M$ mixes in time $O(\nicefrac{N}{\log^6 N})$.

Hence after time $\nicefrac{N}{\log^5 N}$ we have $o(\nicefrac{1}{\log N})$ total variation distance; subtract this from the escape probability of Corollary \ref{corr_escape_prob} to obtain
\[
\mathbb{E}_\rho\left(T_D\right)= \Omega\left(
\frac{\mathbb{E}_\pi\left(T_D\right)}{\log N}
\right)
\]
but then by Lemma \ref{prop_ergodic_return} this implies 
$
\mathbb{E}_\pi\left(T_D\right)=O\left(
N \log N
\right).
$
\end{proof}

Upper bounding meeting gives an easy upper bound for coalescence, as long as we are accepting log factor corrections.

\begin{corollary}\label{cor_super_fast_upper}
When $\frac{\log^7 N}{N} \leq \kappa \leq \frac{1}{\log N}$ and $\beta > 1-2\gamma$
\[
\mathbb{E}\left(
T_{\rm coal}
\right)=
O\left(  N \log^2 N \right)
\]
or when $\kappa \geq \frac{1}{\log N}$
\[
\mathbb{E}\left(
T_{\rm coal}
\right)=
O\left(  N \log N \right).
\]
\end{corollary}

\begin{proof}
In Corollary \ref{cor_large_k_meeting} we bounded meeting from stationarity, but by waiting for a mixing period we see
\[
t_{\rm meet}
:=
\max_{x,y,g}\mathbb{E}_{x,y,g}\left(T_D\right)
=
\mathbb{E}_\pi\left(T_D\right)
+
O\left(
\frac{1}{\kappa} \log N
\right)
\]
so that the worst case initial condition has the same order.

We now consider the coalescence time. Number a walker for each $i \in [N]$ by initial state $W_0^{(i)}=i$, couple them after meeting by following the path of the walker of lowest index, and write $T_D(i,j)$ for the first meeting of walkers $i$ and $j$. This construction produces
\[
T_{\rm coal}
\leq
\max_{j\geq 1}
T_D(1,j)
\implies
\p\left(
T_{\rm coal}>t
\right)\leq
1 \wedge
\sum_{j\geq 1}
\p\left(
T_D(1,j)>t
\right).
\]

By \cite{aldous-fill-2014}, Equation 2.20,
\[
\sum_{j\geq 1}
\p\left(
T_D(1,j)>t
\right)
\leq
N 
\exp\left(1-
\frac{t}{e t_{\rm meet}}
\right)
\]
from which it follows %
\[
\mathbb{E}\left(
T_{\rm coal}
\right)
\leq\int_0^\infty
1 \wedge
N e
\exp\left(-
\frac{t}{e t_{\rm meet}}
\right){\rm d}t
=e(2+\log N)t_{\rm meet}.
\]

This bound applies to every $\kappa \geq \frac{\log^7 N}{N}$, but in the second case of the claim we observe that at those large $\kappa$ values we no longer need to use the giant, via Proposition \ref{prop_subcrit_coal}.
\end{proof}

That result finishes one case of this section. In the other of a faster graph dynamic, we will be interested in finding the correct order without any polylogarithmic correction. First we note a simplifying stochastic upper bound.

\begin{lemma}\label{lem_poisexp}
For any $\mu, \lambda>0$
\[
\operatorname{Pois}(\mu) \preceq \lceil \mu e^\lambda \rceil + \operatorname{Exp}(\lambda).
\]
\end{lemma}

\begin{proof}
To upper bound a discrete distribution by an exponential we must upper bound by the floor which is $\operatorname{Geom}(1-e^{-\lambda})$. 
The ratio between successive Poisson masses at $x-1$ and $x$ is $\nicefrac{\mu}{x}$, which is at least as steep a decay as the geometric when $x \geq \mu e^\lambda$. 
\end{proof}

Now, we need to control the mixing of a system of $N$ walkers sharing a graph, similarly to Corollary \ref{corr_mixing}.

\begin{lemma}\label{lemma_full_mix}
Write $M^{(N)}$ for the Cartesian product of $N$ walkers (without coalescence) and the dynamic graph which they all inhabit. We find when $\nicefrac{1}{\kappa}=\Omega(N)$ that
\[
t_{\rm mix}\left(
M^{(N)}
\right)=
O\left(\frac{1}{\kappa }\log N\right).
\]
\end{lemma}

\begin{proof}
Let $k \leq N$ walkers start in some arbitrary configuration on $[N]$, i.e. each walker with label $i \in [k]$ has some $w_i \in [N]$ with
$
p_i^{(0)}(v)=\delta_{v w_i}.
$ 
The forward distributions $p_i^{(t)}$ are a function of the dynamic graph history, and so is
\[
P^{(t)}(v):=\frac{1}{k} \sum_{i=1}^k p_i^{(t)}.
\]

For $p_1$ we have strong stationary times in Lemma \ref{lem_sep} which bound the separation distance by \cite{wilmer2009markov}, Lemma 6.12 (or indeed the proof of Lemma \ref{lem_sep} was by bounding the separation). Hence, assuming $\kappa=o(1)$, we have
\[
\p\left(
\max
\left|
1-N p_1^{(t)}
\right|
> \frac{1}{5} \right)
\leq e^{-\Omega(t\kappa)}
\]
where the factor $N$ is the constant $\nicefrac{1}{\pi}$.

$P$ is not constructed as the distribution of a Markov chain. However, it has the exact same Markov generator as $p_1$ and so the same separation statement is true for $P$.

Write $\cF_t$ for the $\sigma$-algebra generated by the graph history up to time $t$, and note $P^{t}$ and $p^{t}_i$ are $\cF_t$-measurable. By applying \cite{mcdiarmid1998concentration}, Theorem 2.3(c),
\[
\p\left(
v\text{ occupied at time }t\,
\big| \, \cF_t
\right)
\geq
1-e^{-\frac{1}{2} k P^{(t)}(v)}
\]
and so by applying the separation control at time $t_0=\nicefrac{\mu}{2\kappa}$
\[
\p\left(
v\text{ occupied at time }t_0
\right)
\geq
\left(
1-e^{-\Omega(t_0\kappa)}
\right)
\left(
1-e^{- \nicefrac{2k}{5N}  }
\right)
\geq \frac{k}{4N}
\]
for some sufficiently large $\mu>0$.

Write $\cO(t)$ for the total time spent occupied by all sites in $[N]$ over the interval $[0,t]$, and $\cU(t)$ for the total number of updates seen at occupied sites. An update at time $t$ will influence the future distribution of $\cO$, but nonetheless
\[
\kappa\cO(t)-\cU(t)
\]
is a martingale. Moreover, if we modify to $\tilde{\cO}$ and $\tilde{\cU}$ which stop recording for time $1$ after the first update of each walker this is still true, and we have
\[
\mathbb{E}(\tilde{\cO}(2t_0))\geq N \cdot (t_0-k) \cdot \frac{k}{4N}.
\]

Additionally, the number of updates seen by each walker is dominated by an independent $\operatorname{Pois}(2\kappa t_0)=\operatorname{Pois}(\mu)$. Consider the order statistics of these Poissons ${P_{(1)}<\dots<P_{(k)}}$, which by Lemma \ref{lem_poisexp} are bounded by the order statistics of an exponential. The expected order statistics of an exponential of rate $1$ have the simple expression
\[
\mathbb{E}(E_{(i)})= H_k-H_{k-i}
\]
in terms of harmonic numbers $H_b=\sum_{a=1}^b\nicefrac{1}{a}$, and hence
\[
\mathbb{E}\left(P_{(i)}\right)\leq
\lceil \mu e^\lambda \rceil
+
\frac{1}{\lambda}\left(
 H_k-H_{k-i}
 \right).
\]

So, the expectation of the $j$ largest Poissons is bounded by
\[
\mathbb{E}\left(
\sum_{i=k-j+1}^k
P_{(i)}\right)\leq
j\lceil \mu e^\lambda \rceil
+
\frac{j}{\lambda}
\log
\left(
 \frac{ek}{j}
 \right)
\]
and from this, writing $\cE_j(t)$ for the event that at most $j$ walkers see an update in $[0,t]$, we use the upper bound by independent Poisson processes to infer 
\[
\begin{split}
\mathbb{E}\left(\tilde{\cU}(t)\right)
&=
\mathbb{E}\left(\tilde{\cU}(t));\cE_j(t)\right)
+
\mathbb{E}\left(\tilde{\cU}(t));\cE_j^c(t)\right)\\
&\leq
j\lceil \mu e^\lambda \rceil
+
\frac{j}{\lambda}
\log
\left(
 \frac{ek}{j}
 \right)
+
(j+1+k\mu) \p\left(\cE_j^c\right)\\
\end{split}
\]
which is the desired lower bound on the probability to see update events of Lemma \ref{lem_sep} for $j+1$ walkers. 
Rearranging, what we have exactly is
\[
\begin{split}
\p\left(\cE_j^c(2t_0)\right)
&\geq 
 \frac{\frac{1}{4}\kappa k(t_0-N)
-j\lceil \mu e^\lambda \rceil
-
\frac{j}{\lambda}
\log
\left(
 \frac{ek}{j}
 \right) }{j+1+k\mu}
\\
&\geq 
 \frac{\frac{\mu-1}{4}
-\frac{1}{2}(1+\mu e^\lambda)
-
\frac{1}{2\log \lambda}
\log
\left(
2e
 \right) }{1+\mu}
 \geq 
 \frac{\frac{\mu-1}{4}
-1
-
\frac{1}{2\log \nicefrac{1}{\mu}}
\log
\left(
2e
 \right) }{1+\mu}
 \geq \frac{1}{10}, \\
\end{split}
\]
the final line by setting $j \leq \nicefrac{k}{2}$, then $\lambda=\log\nicefrac{1}{\mu}$, and finally $\mu \geq 6$.

 Achieving a constant proportion of walkers mixed with positive probability every timestep of length $\nicefrac{6}{\kappa}$ is an algorithm that will get to constantly many unmixed walkers in time $O(\frac{1}{\kappa}\log N)$, from where it remains to apply Corollary \ref{corr_mixing} constantly many times until the final walkers and the graph are also stationary.
\end{proof}

The consequent small $\kappa$ case introduces a technical assumption. As this case uses the giant component to meet, we assume that $\beta \geq 3$ to give a giant component of size at least $\nicefrac{9N}{10}$.

\begin{proposition}\label{prop_small_k_coal}
When $\beta \geq 3$, $\alpha > 1$ and $\kappa=\Theta(N^{-\alpha})$
\[
\mathbb{E}\left(
    T_{\rm coal}
    \right)=
    O\left(  \frac{1}{\kappa} \log N \right).
\]
\end{proposition}

\begin{proof}

We build a good set $S$ for $N$ non-coalescing walkers sharing a dynamic graph, i.e. the chain $M^{(N)}$ of Lemma \ref{lemma_full_mix}, defined by properties:
\begin{itemize}
\item $|\sC_{\rm max}|\geq\nicefrac{9N}{10}$;
\item The one-walker mixing of the giant has $t_{\rm mix}(\sC_{\rm max})\leq \log^{17} N$;
\item $\nicefrac{4N}{5}$ walkers are in the set $\sC_{\rm max}$.
\end{itemize}

Note, from Theorem 3.1 of \cite{bollobas07}, $\beta>3$ gives high probability giant proportion at least $0.94$, and Theorem 1.9 of \cite{fernley2022discursive} shows that the mixing of the largest component is of high probability order  $O_{\p}(\log^{16} N)$. For the last condition, conditionally we place at least $\operatorname{Bin}(N,\nicefrac{9}{10})$ stationary walkers in the giant which is concentrated around $\nicefrac{9N}{10}$. Hence $S$ is a high probability set for the stationary chain 
$
\pi(S) \rightarrow 1.
$

So after a mixing period of Lemma \ref{lemma_full_mix} we are with high probability in the set $S$. If so, \emph{freeze} the graph for time $\log^{19} N$. In this time the graph updates $\operatorname{Pois}\left(
 \kappa N\log^{19} N
\right)$ times, so it doesn't update with probability
\[
e^{- N^{1-\alpha}\log^{19} N}\geq 1-N^{1-\nicefrac{\alpha}{2}}.
\]

On the event that the graph stays frozen, after time $\log^{19} N$ each walker is coupled to stationarity with probability
\[
1-e^{-\Omega(\log^2 N)}=1-o\left(
\frac{1}{N}
\right)
\]
and so with high probability we find all (at least $\nicefrac{4N}{5}$) walkers uniform on $\sC_{\rm max}$. 
Recalling Lemma \ref{lem_advanced_birthday}, at a single point in time we then find \[\frac{\left(\nicefrac{4N}{5}\right)^2}{6N}=
\frac{8N}{75}\] coalescences with probability $\nicefrac{1}{73}$. Thus in time \[
O\left(\frac{1}{\kappa }\log N\right)+\log^{19} N=
O\left(\frac{1}{\kappa }\log N\right)
\]
we have $\Omega(N)$ coalescences with positive probability. Whether we see them or not, reset with a further period $O\left(\frac{1}{\kappa }\log N\right)$ and retry.

Because this mixing period allows us to forget which walker has which label, every set of $\Omega(N)$ coalescences is taken uniformly in the $N$ walkers and this altogether we have constructed the Wright--Fisher coalescent on generational timescale $O\left(\frac{1}{\kappa }\log N\right)$.
\end{proof}

{\bf Acknowledgements.}
Thanks to Peter M\"orters for suggesting this problem. 
This research was supported by National Research, Development and Innovation Office grant KKP 137490.

\bibliographystyle{apalike}

\end{document}